\documentclass[11pt,oneside]{amsart}
\usepackage{latexsym, amssymb, stmaryrd}

\newtheorem{thm}{Theorem}[section]
\newtheorem{lemma}[thm]{Lemma}

\newtheorem{cor}[thm]{Corollary}

\newtheorem{fact}[thm]{Fact}
\newtheorem{df}[thm]{Definition}

\newcommand{\BE}{\mathbb {E}}  
\newcommand{\BK}{\mathbb{K}}  
\newcommand{\BN}{\mathbb{N}}

\newcommand{\cu}[1]{\mathcal{#1}}

\newcommand{\bo}[1]{\boldsymbol{#1}}

\newcommand{\sa}[1]{\mathsf{#1}}

\newcommand{\ti}[1]{\widetilde{#1}}

\makeatletter

\def\indsym#1#2{%
  \setbox0=\hbox{$\m@th#1x$}%
  \kern\wd0%
  \hbox to 0pt{\hss$\m@th#1\mid$\hbox to 0pt{$\m@th#1^{#2}$}\hss}%
  \lower.9\ht0\hbox to 0pt{\hss$\m@th#1\smile$\hss}%
  \kern\wd0}

\def\nindsym#1#2{%
  \setbox0=\hbox{$\m@th#1x$}%
  \kern\wd0%
  \hbox to 0pt{\hss$\m@th#1\not$\kern1.4\wd0\hss}
  \hbox to 0pt{\hss$\m@th#1\mid$\hbox to 0pt{$\m@th#1^{\,#2}$}\hss}%
  \lower.9\ht0\hbox to 0pt{\hss$\m@th#1\smile$\hss}%
  \kern\wd0}

\def\dotminussym#1#2{%
  \setbox0=\hbox{$\m@th#1-$}%
  \kern.5\wd0%
  \hbox to 0pt{\hss\hbox{$\m@th#1-$}\hss}%
  \raise.6\ht0\hbox to 0pt{\hss$\m@th#1.$\hss}%
  \kern.5\wd0}

\def \<{\langle}
\def \>{\rangle}
\def \((  {(\!(}
\def \)) {)\!)}

\def \int{\operatorname{int}}

\def \range{\operatorname{range}}
\numberwithin{equation}{section}
\def \l{\llbracket}
\def \rr{\rrbracket}

\allowdisplaybreaks[2]

\begin{document}

\title{Scattered Sentences have Few Separable Randomizations}

\author{Uri Andrews, Isaac Goldbring, Sherwood Hachtman, H. Jerome Keisler, and David Marker}

\address{University of Wisconsin-Madison, Department of Mathematics, Madison,  WI 53706-1388}
\email{andrews@math.wisc.edu}
\urladdr{www.math.wisc.edu/~andrews}
\email{keisler@math.wisc.edu}
\urladdr{www.math.wisc.edu/~keisler}

\address{University of California, Irvine, Department of Mathematics, Irvine, CA, 92697-3875}
\email{isaac@math.uci.edu}
\urladdr{www.math.uci.edu/~isaac}

\address {University of Illinois at Chicago, Department of Mathematics, Statistics, and Computer Science,
Science and Engineering Offices (M/C 249), 851 S. Morgan St., Chicago, IL 60607-7045, USA}
\email{marker@uic.edu}
\urladdr{www.math.uic.edu/~marker}
\email{hachtma1@uic.edu}
\urladdr{www.math.uic.edu/~shac}

\thanks{The work of Andrews was partially supported by NSF grant DMS-1600228.  The work of Goldbring was partially supported by NSF CAREER grant DMS-1708802.}

\begin{abstract}
In the paper \emph{Randomizations of Scattered Sentences}, Keisler showed that if Martin's axiom for aleph one holds, then every scattered sentence has few separable
randomizations, and asked whether the conclusion could be proved in ZFC alone.  We show here that the answer is ``yes''.  It follows that the absolute Vaught conjecture
holds if and only if every $L_{\omega_1\omega}$-sentence with few separable randomizations has  countably many countable models.
\end{abstract}

\maketitle

\section{Introduction}

This note answers a question posed in the paper [K2], and results from a discussion following a lecture by Keisler at the Midwest Model Theory meeting in Chicago on April 5, 2016.

Fix a countable first order signature $L$.  A sentence $\varphi$ of the infinitary logic $L_{\omega_1\omega}$ is \emph{scattered} if there is
no countable fragment $L_A$ of $L_{\omega_1\omega}$ such that $\varphi$ has a perfect set of countable models that are not $L_A$-equivalent.
Scattered sentences were introduced by Morley [M], motivated by Vaught's conjecture.  The absolute form of Vaught's conjecture for an $L_{\omega_1\omega}$-sentence
$\varphi$ says that if $\varphi$ is scattered then $\varphi$ has  countably many (non-isomorphic) countable models
\footnote{Here, countable means of cardinality at most $\aleph_0$.}.

In continuous logic, the \emph{pure randomization theory} $P^R$
(from [BK]) is a theory whose signature $L^R$ has a sort $\BK$ for random elements and a sort $\BE$ for events.  For each formula $\theta(\cdot)$ of $L$ with $n$ free variables,
$L^R$ has a function symbol $\l\theta(\cdot)\rr$ of sort $\BK^n\to\BE$ for the event at which $\theta(\cdot)$ is true.  $L^R$ also has
Boolean operations $\sqcup,\sqcap,\neg$ in the event sort, a predicate $\mu$ from events to $[0,1]$, and
distance predicates $d_\BK, d_\BE$ for each sort.  The set of axioms for $P^R$ is recursive in $L$.  It insures that the functions $\l\theta(\cdot)\rr$ respect validity, connectives,
and quantifiers, that each event is equal to the set where some pair of random elements agree,
and that $\mu$ is an atomless probability measure on the events. There are also axioms that define $d_\BK$ and $d_\BE$ in the natural way.

Pre-models of $P^R$ are called \emph{randomizations}, and models of $P^R$ are called \emph{complete randomizations}.
In Theorem 5.1 of [K2] (stated as Fact \ref{f-infinitary-model} below), it is shown that in a complete separable randomization, there is a unique mapping $\l\cdot\rr$ from $L_{\omega_1\omega}$-sentences
to events that respects validity, countable connectives, and quantifiers.
A \emph{separable randomization of} an $L_{\omega_1\omega}$-sentence $\varphi$ is a separable  randomization whose completion satisfies $\mu(\l\varphi\rr)=1$.
Intuitively, in a separable randomization of $\varphi$, a random element is obtained by randomly picking an
element of a random countable model of $\varphi$, with respect to some underlying probability space.  An especially simple kind of randomization of $\varphi$, called a
\emph{basic randomization}, has random elements picked from some fixed countable family of countable models of $\varphi$, with the underlying probability space being the Lebesgue
measure on the unit interval.  $\varphi$ is said to have \emph{few separable randomizations} if every complete separable randomization of $\varphi$ is isomorphic to a basic randomization.

The main results of [K2] are: If an $L_{\omega_1\omega}$-sentence $\varphi$ has  countably many countable models, then $\varphi$ has few separable randomizations.
If $\varphi$ has few separable randomizations,
then $\varphi$ is scattered.  If Martin's axiom for $\aleph_1$ holds and $\varphi$
is scattered, then $\varphi$ has few separable randomizations.  [K2] asks whether the conclusion of this last result can be proved in ZFC.  Here we will show that the answer to
that question is ``yes''.  The idea will be to use the Shoenfield absoluteness theorem  to eliminate the use of Martin's axiom.

The results in the preceding paragraph show that being scattered is equivalent to having few separable randomizations.   The absolute Vaught conjecture for $\varphi$ says that 
if $\varphi$ is scattered then $\varphi$ has countably many countable models.  Thus the absolute Vaught conjecture is equivalent to the property that
having few separable randomizations implies having countably many countable models.

\section{Background} \label{s-background}
 We refer to [BBHU] for background in continuous logic, [J] for background
on absoluteness and Martin's axiom, and [K1] for background on $L_{\omega_1\omega}$.  We assume throughout that  $\varphi$ is an
$L_{\omega_1\omega}$-sentence that implies $(\exists x)(\exists y) x\ne y$.
We will not need the formal statement of the axioms of $P^R$, or the formal definition of $\l\psi(\cdot)\rr$ for $L_{\omega_1\omega}$-formulas $\psi(\cdot)$.
In this section we will state the definitions and results from [K2] that we will need.

Given two pre-structures $\cu N$ and $\cu P $ with signature $L^R$, an \emph{isomorphism}  $h\colon\cu N\to\cu P$
is a mapping from $\cu N$ into $\cu P$ such that $h$ preserves the truth values of all formulas of $L^R$, and every element of $\cu P$ is at distance
zero from some element of $h(\cu N)$.  We call $\cu P$ a \textbf{reduction of} $\cu N$ if $\cu P$ is obtained from $\cu N$ by identifying elements at distance zero,
and call $\cu P$ a \textbf{completion of} $\cu N$ if $\cu P$ is a structure obtained from a reduction of $\cu N$ by completing the metrics.
Up to isomorphism, every pre-structure has a unique reduction and completion.
The mapping that identifies elements at distance zero is called the \textbf{reduction mapping}, and is an isomorphism from a pre-structure onto its reduction.

The axioms of $P^R$ have the following consequences:
$$d_\BK(\bf f,\bo g)=\mu(\l\bo f \ne \bo g\rr),\quad d_\BE(\sa A,\sa B)=\mu(\sa A\triangle \sa B).$$
$$\mu(\l(\exists x)(\exists y)x\ne y\rr)=1.$$
By the latter, every separable randomization is a separable randomization of $(\exists x)(\exists y)x\ne y.$
Since $P^R$ has axioms saying that the functions $\l\theta(\cdot)\rr$ for first order $\theta$ respect connectives, and that every event
is equal to $\l\bo a=\bo b\rr$ for some $\bo a, \bo b$, it follows that:

\begin{fact} \label{f-isomorphism}
Suppose $\cu N=(\cu K,\cu B)$ and $\cu N'=(\cu K',\cu B')$ are models of $P^R$, $h$ maps  $\cu K$ onto $\cu K'$, and
$$\cu N\models\mu(\l\theta(\vec{\bo a})\rr)\ge r \Leftrightarrow\cu N'\models\mu^{\cu N'}(\l\theta(h\vec{\bo a})\rr)\ge r$$
 for all first order $\theta$, tuples $\vec{\bo a}$ in $\cu K$, and rational $r$.  Then $h$ can be extended to a unique isomorphism from $\cu N$ onto
 $\cu N'$.
\end{fact}

The simplest examples of randomizations are the Borel randomizations, defined as follows.
Let $\cu L$ be the family of Borel subsets of $[0,1)$ and $\lambda$ be the restriction of Lebesgue measure to $\cu L$.

\begin{df}
The \emph{Borel randomization} of a model $\cu M\models(\exists x)(\exists y)x\ne y$ is the structure $(\cu M^{\cu L},\cu L)$ of sort $L^R$ where $\cu M^{\cu L}$ is the set of all functions $\bo f\colon[0,1)\to M$ with countable range
such that $\{t\mid \bo f(t)=a\}\in\cu L$ for each $a\in M$, $\cu L$ has the usual Boolean operations, $\mu$ is interpreted by $\lambda$, and
$$\l\theta(\vec{\bo f})\rr=\{t\mid \cu M\models \theta(\vec{\bo f}(t))\}.$$
\end{df}

A basic randomization of $\varphi$ is formed by ``gluing together'' countably many Borel randomizations of countable models of $\varphi$.

\begin{df} \label{def-basic-randomizations}  Suppose that
\begin{itemize}
\item $[0,1)=\bigcup_{n} \sa B_n$
is a partition of $[0,1)$ into countably many Borel sets of positive measure;
\item for each $n$, $\cu M_n$ is a countable model of $\varphi$;
\item $\prod_{n} \cu M_n^{\sa B_n}$ is the set of all functions $\bo f\colon[0,1)\to\bigcup_{n} M_n$
such that for all $n$,
$$(\forall t\in \sa B_n)\bo f(t)\in M_n\mbox{ and } (\forall a\in M_n) \{t\in \sa B_n\mid \bo f(t)=a\}\in\cu L;$$
\item  $(\prod_{n} \cu M_n^{\sa B_n},\cu L)$ has the usual Boolean operations, $\mu$ is interpreted by $\lambda$, and
the $\llbracket \theta(\cdot)\rrbracket$ functions are
 $$\llbracket \theta(\vec{\bo f})\rrbracket=\bigcup_{n}\{t\in\sa B_n\mid\cu M_n\models\theta(\vec{\bo f}(t))\}.$$
\end{itemize}
$(\prod_{n} \cu M_n^{\sa B_n},\cu L)$ is called a \emph{basic randomization} of $\varphi$.
\end{df}

\begin{fact}  \label{f-infinitary-model}  (Theorem 5.1 in [K2])
Let $\cu P=(\cu K,\cu E)$ be a complete separable randomization, and let $\Psi_n$ be the class of
$L_{\omega_1\omega}$ formulas with $n$ free variables.  There is a unique family of
functions $\llbracket\psi(\cdot)\rrbracket^\cu P$, $\psi\in\bigcup_n\Psi_n$, such that:
\begin{itemize}
\item[(i)] When $\psi\in\Psi_n$,
$\llbracket\psi(\cdot)\rrbracket^\cu P\colon\cu K^n\to\cu E.$
\item[(ii)] When $\psi$ is a first order formula, $\llbracket\psi(\cdot)\rrbracket^\cu P$ is the
usual event function for the structure $\cu P$.
\item[(iii)] $\llbracket \neg\psi(\vec{\bo f})\rrbracket^\cu P=\neg\llbracket \psi(\vec{\bo f})\rrbracket^\cu P$.
\item[(iv)] $\llbracket (\psi_1\vee\psi_2)(\vec{\bo f})\rrbracket^\cu P=
\llbracket \psi_1(\vec{\bo f})\rrbracket^\cu P\sqcup\llbracket \psi_2(\vec{\bo f})\rrbracket^\cu P.$
\item[(v)] $\llbracket\bigvee_k\psi_k(\vec{\bo f})\rrbracket^\cu P=\sup_k\llbracket\psi_k(\vec{\bo f})\rrbracket^\cu P.$
\item[(vi)] $\llbracket(\exists u)\theta(u,\vec{\bo f})\rrbracket^\cu P=
\sup_{\bo g\in\cu K}\llbracket\theta(\bo g,\vec{\bo f})\rrbracket^\cu P.$
\end{itemize}
Moreover, for each $\psi\in\Psi_n$, the function $\llbracket\psi(\cdot)\rrbracket^\cu P$ is Lipschitz continuous with bound one, that is,
for any pair of $n$-tuples $\vec{\bo f}, \vec{\bo h}\in\cu K^n$ we have
$$ d_\BE(\llbracket\psi(\vec{\bo f})\rrbracket^\cu P,\llbracket\psi(\vec{\bo h})\rrbracket^\cu P)\le\sum_{m<n} d_\BK(\bo f_m,\bo h_m).$$
\end{fact}

\begin{df} Let $\cu N$ be a separable randomization with completion $\cu P$, and $\varphi$ be an $L_{\omega_1\omega}$-sentence.  We write
$$\mu^{\cu N}(\l\varphi\rr)=\mu^{\cu P}(\l\varphi\rr)=\mu(\l\varphi\rr^{\cu P}).$$
If $\mu^{\cu N}(\l\varphi\rr)=1$, we say that $\cu N$ is a \emph{randomization of} $\varphi$.

We say that $\varphi$ has \emph{few separable randomizations} if every complete separable randomization of $\varphi$ is isomorphic to a basic randomization of $\varphi$.
\end{df}

\begin{fact}  \label{f-basic}  ([K2], Lemma 4.3 and Theorem 4.6.)
Every basic randomization of $\varphi$ is isomorphic to its reduction, which is a complete separable randomization of $\varphi$
(and thus a model of $P^R$).
\end{fact}

\begin{fact} \label{f-9.4}  (Lemma 9.4 in [K2])  Let $(\prod_{j\in J} \cu M_j^{\sa B_j},\cu L)$ be a basic randomization.  For each $j\in J$, let $\delta_j$ be a
Scott sentence of $\cu M_j$.  Then for each complete separable randomization $\cu P$ of $\varphi$, the following are equivalent.
\begin{itemize}
\item $\cu P$ is isomorphic to $(\prod_{j\in J} \cu M_j^{\sa B_j},\cu L)$.
\item $\mu^{\cu P}(\l\delta_n\rr)=\lambda(\sa B_j)$ for each $j\in J$.
\end{itemize}
\end{fact}

\begin{fact}  \label{f-9.5}  (Lemma 9.5 in [K2])
$\varphi$ has few separable randomizations if and only if for every complete separable randomization (or every countable randomization)
$\cu N$ of $\varphi$ there is a Scott sentence $\delta$ such that $\mu^{\cu N}(\l\delta\rr)>0.$
\end{fact}

\begin{fact}  \label{f-10.1}  (Theorem 10.1 in [K2]).
If $\varphi$ has few separable randomizations, then $\varphi$ is scattered.
\end{fact}

\begin{fact}  \label{f-Martin}  (Theorem 10.3 in [K2]).
Assume that Lebesgue measure is $\aleph_1$-additive (e.g. assume that MA$(\aleph_1)$ holds).  Then every scattered sentence  has few separable randomizations.
\end{fact}

Question 11.4 in [K2] asks whether or not the conclusion of Fact \ref{f-Martin} can be proved in ZFC.

\section{The Main Result}

We will prove the following theorem, which answers Question 11.4 in [K2] affirmatively.

\begin{thm}  \label{t-main}
Every scattered sentence has few separable randomizations.
\end{thm}

Fact \ref{f-10.1} and Theorem \ref{t-main} give us the following two corollaries.

\begin{cor} \label{c-scattered-char}
A sentence of $L_{\omega_1\omega}$ is scattered if and only if it has few separable randomizations.
\end{cor}

\begin{cor}
For each $L_{\omega_1\omega}$-sentence $\varphi$, the following are equivalent.

\begin{itemize}
\item[(i)] The absolute Vaught conjecture for $\varphi$ holds.
\item[(ii)]  If $\varphi$ has few separable randomizations, then $\varphi$ has countably many countable models.
\end{itemize}
\end{cor}

Note that each countable pre-structure $\cu N=(\cu K,\cu B)$ in the signature $L^R$ can be coded in a natural way by a first order structure with universe $\BN$ and a countable
signature indexed by $\BN$.  In particular, the function $\mu\colon\cu B\to[0,1]$ can be coded by the set of $(e,m,n)\in\BN^3$ such that $e$ codes an event $\sa E$ and
$m/n\le\mu(\sa E).$

Let $\cu A$ be the set of subsets of $[0,1)$ that are finite unions of intervals with rational endpoints.
Given a countable model $\cu M$ of $(\exists x)(\exists y)x\ne y$ with countable signature $L$, let $\cu M^{\cu A}$ be the set of functions $f\colon[0,1)\to\cu M$ with finite range such that
for each $a\in\cu M$, $f^{-1}(a)\in\cu A$.  Let $\ti{\cu M}$ be the completion of $(\cu M^{\cu A},\cu A)$.
$\ti{\cu M}$ is isomorphic to the Borel randomization  $(\cu M^{\cu L},\cu L)$ of $\cu M$.
$\cu A$, $\cu M$, and $\cu M^{\cu A}$ are countable and can be coded in the natural way by subsets of $\BN$.

\begin{lemma} \label{L2} Let $\cu N=(\cu K,\cu B)$ be a countable randomization with a coding.
Then the statement (S) below is equivalent (in ZFC) to a $\Sigma^1_1$ formula with parameter $\cu N$.
\begin{itemize}
\item[(S)] There exists a Scott sentence $\delta$ such that $\mu^{\cu N}(\l\delta\rr)>0$.
\end{itemize}
\end{lemma}

\begin{proof}  For each event $\sa C$ in the completion of $\cu N$ such that $\mu(\sa C)>0$, let $\mu|\sa C$ be the conditional measure such that
$$(\mu |\sa C)(\sa E)=\mu(\sa E\sqcap\sa C)/\mu(\sa C),$$
and let $\cu N |\sa C$ be the completion of the pre-structure obtained from $\cu N$ by replacing $\mu$ by $\mu|\sa C.$
We first show that (S) is equivalent to the following statement.

\begin{itemize}
\item[(S')]  There exists a countable model $\cu M$ of $(\exists x)(\exists y)x\ne y$ and an event $\sa C$ in the completion of $\cu N$ such that
$\mu(\sa C)>0$ and $\cu N|\sa C\cong\ti{\cu M}$.
\end{itemize}

Assume (S). Let $\delta$ be a Scott sentence $\delta$ such that $\mu^{\cu N}(\l\delta\rr)>0$.
 Let $\sa C=\l\delta\rr$, which is an event of positive measure  in the completion of $\cu N.$
Then $\mu^{\cu N|\sa C}(\l\delta\rr)=1$, so  $\cu N | \sa C$ is a separable randomization of $\delta$.  Let $\cu M$ be a countable  model of $\delta$.   By Fact \ref{f-9.4},
we have $\cu N | \sa C\cong \ti{\cu M}$, so (S') holds.

Now assume (S'). By Scott's theorem, $\cu M$ has a Scott sentence $\delta$. Then by Fact \ref{f-9.4}, $\mu^{\ti{\cu M}}(\l\delta\rr)=1$, so
$$1=\mu^{\cu N |\sa C}(\l\delta\rr)=\mu^{\cu N}(\sa C\sqcap\l\delta\rr)/\mu^{\cu N}(\sa C).$$
Hence
$$\mu^{\cu N}(\l\delta\rr)\ge\mu^{\cu N}(C\sqcap\l\delta\rr)=\mu^{\cu N}(\sa C)>0,$$
so (S) holds.

We now show that (S') is equivalent to the following statement.

\begin{itemize}
\item[(S'')] There exists a countable coded structure $\cu M$ with at least $2$ elements, a sequence $\sa B\colon\BN\to\cu B$, and double sequences
$\alpha\colon\BN\times\BN\to \cu M^{\cu A}$, $\beta\colon\BN\times\BN\to\cu K$ such that
\begin{itemize}
\item[(a)]   $\sa B$ is  Cauchy convergent in $d_{\BE}$, and $\lim_{n\to\infty} \mu(\sa B_n)>0$.
\item[(b)]  For each $m\in\BN$, $\< \alpha_{m,n}\mid n\in \BN\>$ and $\< \beta_{m,n}\mid n\in \BN\>$ are Cauchy convergent in $d_{\BK}$.
\item[(c)] For each $x\in\cu M^{\cu A}$, there exists $m_x\in\BN$ such that $\alpha_{m_x,n}=x$ for all $n\in\BN$, and for each
 $y\in\cu K$, there exists $m_y\in\BN$ such that $\beta_{m_y,n}=y$ for all $n\in\BN$.
\item[(d)] For each $L$-formula $\psi(v_1,\ldots,v_k)$,
$$ \lim_{n\to\infty}\mu^{\ti{\cu M}}(\l\psi(\alpha_{1,n},\ldots,\alpha_{k,n})\rr)=
\lim_{n\to\infty}\mu^{\cu N}(\l\psi(\beta_{1,n},\ldots,\beta_{k,n})\rr\sqcap \sa B_n)/\mu^{\cu N}(\sa B_n).$$
\end{itemize}
\end{itemize}

In (S''), $\cu N$ and $\cu M$ are coded structures, so (S'') is clearly $\Sigma^1_1$ with parameter $\cu N$.

The functions $\l\psi(\cdot)\rr$ are uniformly continuous in each model of $P^R$.
Whenever (a) and (b) hold, for each $m,n\in\BN$ the reduction maps send $\alpha_{m,n}$ to an element $\alpha''_{m,n}$ of $\ti{\cu M}$,
and $\beta_{m,n}$ to an element $\beta''_{m,n}$ of $\cu N|\sa C$, and the limits
 $\alpha'_m=\lim_{n\to\infty}\alpha''_{m,n}$ in $\ti{\cu M}$ and $\beta'_m=\lim_{n\to\infty}\beta''_{m,n}$ in $\cu N|\sa C$ exist.
Therefore, (a) and (b) imply that for each $L$-formula $\psi(v_1,\ldots,v_k)$,
\begin{equation} \label{eqM}
\mu^{\ti{\cu M}}(\l\psi(\alpha'_1,\ldots,\alpha'_k)\rr)=\lim_{n\to\infty}\mu^{\ti{\cu M}}
(\l\psi(\alpha_{1,n},\ldots,\alpha_{k,n})\rr)
\end{equation}
and
\begin{equation} \label{eqN}
 \mu^{\cu N|\sa C}(\l\psi(\beta'_1,\ldots,\beta'_k)\rr)=\lim_{n\to\infty}\mu^{\cu N}(\l\psi(\beta_{1,n},\ldots,\beta_{k,n})\rr\sqcap \sa B_n)/\mu^{\cu N}(\sa B_n).
\end{equation}

We next assume that (S') holds for some $\cu M$ and $\sa C$, and prove (S'').  We may take $\cu M$ to be a coded structure, and let $h$ be
an isomorphism  from  $\cu N | \sa C$ to  $\ti{\cu M}$.
We may choose mappings $\alpha'$ from $\BN$ into $\ti{\cu M}$ and $\beta'$ from $\BN$ into $\cu N|\sa C$
such that $ \range(\alpha'), \range(\beta')$ contain the images of $\cu M^\cu A$ and $\cu K$ under the reduction maps, and  $\alpha'_n=h(\beta'_n)$ for each $n\in\BN$.
Then for each $L$-formula $\psi(v_1,\ldots,v_k)$,
\begin{equation}  \label{eq2}
 \mu^{\ti{\cu M}}(\l\psi(\alpha'_1,\ldots,\alpha'_k)\rr)=
\mu^{\cu N|\sa C}(\l\psi(\beta'_1,\ldots,\beta'_k)\rr).
\end{equation}
One can choose a sequence $\sa B\colon\BN\to\cu B$, and double sequences
$\alpha\colon\BN\times\BN\to \cu M^{\cu A}$, $\beta\colon\BN\times\BN\to\cu K$ such that (c) holds,
the reduction of $\sa B_n$ converges to $\sa C$,
and for each $m\in \BN$ the reductions of $\alpha_{m,n}$ and $\beta_{m,n}$ converge to $\alpha'_m$ and $\beta'_m$ respectively.
Then conditions (a) and (b) hold, so (\ref{eqM}) and (\ref{eqN}) hold for each $L$-formula $\psi(v_1,\ldots,v_k)$,
By (\ref{eq2}), condition (d) holds, and hence (S'') holds.

Finally, we assume (S'') and prove (S').  
Let $\sa C=\lim_{n\to\infty} \sa B_n$ in the completion of $\cu B$.
Since (a) and (b) hold, (\ref{eqM}) and (\ref{eqN}) hold for each $L$-formula $\psi(v_1,\ldots,v_k)$.  Then by (d), (\ref{eq2}) holds for every $\psi$.
By (c), $\range(\alpha')\supseteq\cu M^{\cu A}$ and $\range(\beta')\supseteq\cu K$.
Therefore $\range(\alpha')$ is dense in the $\BK$-sort of $\ti{\cu M}$, and $\range(\beta')$ is dense in the $\BK$-sort of $\cu N|\sa C$.
Hence every element of $\ti{\cu M}$ of sort $\BK$ is equal to $\lim_{k\to\infty} \alpha'_{m_k}$ for some sequence $(m_0,m_1,\ldots)\in\BN^{\BN}$, and similarly for $\cu N|\sa C$ and $\beta'$.
Since $d_\BK(\bo a,\bo b)=\mu(\l \bo a\ne \bo b\rr)$ in any model of $P^R$, $\lim_{k\to\infty} \alpha'_{m_k}$ exists in $\ti{\cu M}$
if and only if $\lim_{k\to\infty} \beta'_{m_k}$ exists in $\cu N|\sa C$.
Whenever $\lim_{k\to\infty} \alpha'_{m_k}$ exists in $\ti{\cu M}$, let $h(\lim_{k\to\infty} \alpha'_{m_k})=\lim_{k\to\infty} \beta'_{m_k}$.
Then $h$ maps the $\BK$-sort of $\ti{\cu M}$ onto the $\BK$-sort of $\cu N|\sa C$.
Since (\ref{eq2}) holds and the functions $\l\psi(\cdot)\rr$ are uniformly continuous in $\ti{\cu M}$ and $\cu N|\sa C$,
$$\mu^{\ti{\cu M}}(\l \psi(\vec{\bo a})\rr)=\mu^{\cu N|\sa C}(\l \psi(h\vec{\bo a})\rr)$$
for each $L$-formula $\psi$ and tuple $\vec{\bo a}$ of sort $\BK$ in $\ti{\cu M}$.  Therefore by Fact \ref{f-isomorphism}, $h$ can be extended to an isomorphism from $\ti{\cu M}$ onto $\cu N | \sa C$.  This proves (S').
\end{proof}

By a \emph{transitive model} of a set of sentences $Z$ we mean a transitive set $V$ such that $(V,\in)\models Z$.
It is well known that there is a finite subset ZFC$_0$ of the set of axioms of ZFC such that
the Shoenfield absoluteness theorem holds for all transitive models of ZFC$_0$.  Assume hereafter that ZFC$_0$ is
a finite subset of ZFC with that property, and also that ZFC$_0$ implies every result stated in Section \ref{s-background},
Lemma \ref{L2} above, and every consequence of ZFC that is used in the proofs of Lemmas \ref{L1} and \ref{l-main} below.

\begin{lemma}  \label{L1}
Let $V, V[G]$ be transitive models of ZFC$_0$ such that the signature $L$ is in $V$, and $V\subseteq V[G]$.  Suppose that in $V$ it is true that
$\varphi$ is an $L_{\omega_1\omega}$-sentence and $\cu N=(\cu K,\cu B)$ is a countable randomization.   Then
in $V[G]$ it is also true that $\varphi$ is an $L_{\omega_1\omega}$-sentence and $\cu N=(\cu K,\cu B)$ is a countable randomization,
 and $\mu^{\cu N}(\l\varphi\rr)$ has the same value in $V$ as in $V[G]$.  Hence
$$ V\models \mbox{$\cu N$  is a countable randomization of } \varphi$$
if and only if
$$ V[G]\models \mbox{$\cu N$  is a countable randomization of } \varphi.$$
\end{lemma}

\begin{proof}   It is easily proved using induction on the complexity of formulas that
$$ V[G]\models \varphi \mbox{ is a an } L_{\omega_1\omega}\mbox{-sentence}.$$
Since the set of axioms of $P^R$ is recursive in $L$, the property of being a countable randomization is $\Sigma_1$, and hence
$$ V[G]\models \mbox{$\cu N$  is a countable randomization}.$$

Let $\cu P$ be the completion of $\cu N$ in $V$, and $\cu Q$ be the completion of $\cu N$ in $V[G]$.
In $V[G]$, $\cu P$ is a separable randomization that is not necessarily complete, and $\cu Q$ is the completion of
 $\cu N$ and also the completion of $\cu P$.  For each $L_{\omega_1\omega}$-formula $\psi(\cdot)$ in $V$, let $\l\psi(\cdot)\rr^{\cu P}$ be
 the function obtained by applying Fact \ref{f-infinitary-model} to $\cu P$  in $V$, and let $\l\psi(\cdot)\rr^{\cu Q}$ be
the function obtained by applying Fact \ref{f-infinitary-model} to $\cu Q$ in $V[G]$.
Using Conditions (i)--(vi) of Fact \ref{f-infinitary-model}, we show by induction on complexity that for every
$L_{\omega_1\omega}$-formula $\psi(\cdot)$ in $V$ and tuple $\vec{\bo f}$ in the reduction of $\cu K$,
$\l\psi(\vec{\bo f})\rr^{\cu P}=\l\psi(\vec{\bo f})\rr^{\cu Q}.$
The base step for first order formulas and the steps for negation and finite disjunction are easy.

Countable disjunction step:  Let $\psi=\bigvee_k\psi_k$, and suppose  $\vec{\bo f}$ is in the reduction of $\cu K$ and that
$\l\psi_k(\vec{\bo f})\rr^{\cu P}=\l\psi_k(\vec{\bo f})\rr^{\cu Q}$
holds for each $k\in\BN$.  Let $\psi'_k=\bigvee_{n\le k}\psi_n$.  Then $\l\psi'_k(\vec{\bo f})\rr^{\cu P}=\l\psi'_k(\vec{\bo f})\rr^{\cu Q}$
for each $k\in\BN$, and
$$\l\psi(\vec{\bo f})\rr^{\cu P}=\lim_{k\to\infty} \l\psi'_k(\vec{\bo f})\rr^{\cu P}=\lim_{k\to\infty} \l\psi'_k(\vec{\bo f})\rr^{\cu Q}=\l\psi(\vec{\bo f})\rr^{\cu Q}.$$

Existential quantifier step:  Let $\psi(\vec u)=(\exists v)\theta(\vec u,v)$ and suppose that
$\l\theta(\vec{\bo f},\bo g)\rr^{\cu P}=\l\theta(\vec{\bo f},\bo g)\rr^{\cu Q}$  for all $\vec{\bo f}, \bo g$ in the reduction of $\cu K$.
Since the reduction of $\cu K$ is dense in the sort $\BK$ parts of both $\cu P$ and $\cu Q$, and the functions $\l\theta(\cdot)\rr^{\cu P}$ and
$\l\theta(\cdot)\rr^{\cu Q}$ are both Lipschitz continuous with bound $1$ by Fact \ref{f-infinitary-model}, it follows that
$\l\psi(\vec{\bo f})\rr^{\cu P}=\l\psi(\vec{\bo f})\rr^{\cu Q}.$  This completes the induction.

Every event in $\cu P$ has the same measure in $V$ as in $V[G]$.
In particular, for the sentence $\varphi$, the measure of $\l\varphi\rr^{\cu P}$
is the same in $V$ as in $V[G]$.  We have
$$V\models\mu^{\cu N}(\l\varphi\rr)=\mu(\l\varphi\rr^{\cu P})$$
and
$$V[G]\models\mu^{\cu N}(\l\varphi\rr)=\mu(\l\varphi\rr^{\cu Q})=\mu(\l\varphi\rr^{\cu P}).$$
Therefore $\mu^{\cu N}(\l\varphi\rr)$ has the same value in $V$ as in $V[G]$.
\end{proof}

 Lemma \ref{L1} can also be proved by using the continuous analogue of the infinitary logic $L_{\omega_1\omega}$.  Lemma 5.18
 in the paper [EV] shows that for any metric structure $\cu P$ and continuous infinitary sentence $\Theta$ in $V$,
 the value of $\Theta$ in $\cu P$ computed in $V$ is the same as the value computed in $V[G]$.  Using Fact \ref{f-infinitary-model},
 one can find a continuous infinitary sentence $\Theta$ that has the same value as $\mu(\l\theta\rr^{\cu P})$ in any
 complete separable randomization $\cu P$, and then use Lemma 5.18 in [EV] to get Lemma \ref{L1}.

\begin{lemma} \label{l-main}  In any countable transitive model $V$ of ZFC$_0$, it is true that
every scattered sentence has few separable randomizations.
\end{lemma}

\begin{proof}
  By the result of Solovay and Tennenbaum, there is a countable transitive model $V[G]$ of ZFC$_0$ with the same ordinals as $V$ such that $V\subseteq V[G]$ and Martin's Axiom for $\aleph_1$ holds in $V[G]$.  Suppose that in $V$ it is true that $\varphi$ is a scattered sentence,  $\cu N$ is a countable
randomization  with a coding, and $\mu^{\cu N}(\l\varphi\rr)=1.$

We now work in $V[G]$, and prove the statement (S) of Lemma \ref{L2}.
The property of being a scattered sentence is $\Pi^1_2$, so by the Shoenfield absoluteness theorem,  $\varphi$ is a still scattered sentence.
By Lemma \ref{L1}, $\cu N$ is still a countable randomization with $\mu^{\cu N}(\l\varphi\rr)=1.$
So the completion of $\cu N$ is a complete separable randomization of $\varphi$.  By
Fact \ref{f-Martin} and Martin's axiom, $\varphi$ has few separable randomizations.
By Fact \ref{f-9.5}, there exists a Scott sentence $\delta$ such that $\mu^{\cu N}(\l\delta\rr)>0$, so (S) holds.

By Lemma \ref{L2} and the Shoenfield absoluteness theorem (or even the weaker Mostowski absoluteness theorem),
(S) also holds in $V$.  So by Fact \ref{f-9.5}, it is true in $V$ that $\varphi$ has few separable randomizations.
\end{proof}

\begin{proof} (Proof of Theorem \ref{t-main})  The following argument is well-known, and is included for completeness.
Let $\eta$ be the sentence in the vocabulary of ZFC that says that every scattered sentence has few separable randomizations.
Assume $\neg\eta$.  By the reflection theorem, ZFC$_0\cup\{\neg\eta\}$ has a transitive model.  By the downward L\"{o}wenheim-Skolem
theorem and the Mostowski collapsing lemma, ZFC$_0\cup\{\neg\eta\}$ has a countable transitive model.  This contradicts Lemma \ref{l-main},
so $\eta$ holds.
\end{proof}

\section*{References}


[BBHU]  Ita\"i Ben Yaacov, Alexander Berenstein,
C. Ward Henson and Alexander Usvyatsov. Model Theory for Metric Structures.  In Model Theory with Applications to Algebra and Analysis, vol. 2,
London Math. Society Lecture Note Series, vol. 350 (2008), 315-427.

[EV]    Christopher Eagle and Alessandro Vignati.  Saturation and Elementary Equivalence of $C^*$-Algebras.  ArXiv:1406.4875v4 (2015).

[J] Thomas Jech.  Set Theory.  Springer-Verlag 2003.

[K1]  H. Jerome Keisler.  Model Theory for Infinitary Logic.  North-Holland 1971.

[K2]  H. Jerome Keisler.  Randomizations of Scattered Sentences.  To appear in the forthcoming volume ``Beyond First Order Model Theory'', edited by Jose Iovino, CRC Press.
Also available online at math.wisc.edu/$\sim$keisler.

\end{document}